\documentclass[11pt,reqno]{amsart}
\usepackage{amsthm,amssymb,amsmath}
\usepackage{xcolor}
\usepackage[foot]{amsaddr}

\newtheorem{theorem}{Theorem}
\newtheorem*{theorem*}{Theorem}
\newtheorem{lemma}{Lemma}
\newtheorem*{lemmaA}{Lemma A}

\theoremstyle{definition}

\newtheorem*{definition*}{\sc Definition}

\newtheorem{example}{\bf Example}
\newtheorem{remark}{\bf Remark}
\newtheorem*{remark*}{\bf Remark}

\newtheorem*{example*}{\bf Example}

\newcommand{\loc}{{\rm loc}}

\newcommand{\Real}{{\rm Re}}

\newcommand{\dist}{\mbox{dist}}

\newcommand{\const}{{\rm const}}
\newcommand{\cl}{{\rm clos}}

\newcommand{\clos}{{\rm clos}}

\begin{document}

\title{Brownian motion with general drift}

\author{D.\,Kinzebulatov and Yu.\,A.\,Semenov} 

\address{Universit\'{e} Laval, D\'{e}partement de math\'{e}matiques et de statistique, pavillon Alexandre-Vachon 1045, av. de la M\'{e}decine, Qu\'{e}bec, PQ, G1V 0A6, Canada}

\email{damir.kinzebulatov@mat.ulaval.ca}

\thanks{\sf http://archimede.mat.ulaval.ca/pages/kinzebulatov}

\address{University of Toronto, Department of Mathematics, 40 St.\,George Str., Toronto, ON, M5S 2E4, Canada}

\email{semenov.yu.a@gmail.com}



\begin{abstract}
We construct and study the weak solution to stochastic differential equation $dX(t)=-b(X(t))dt+\sqrt{2}dW(t)$, $X_0=x$, for every $x \in \mathbb R^d$, $d \geq 3$, with $b$ in the class of weakly form-bounded vector fields, containing, as proper subclasses, a sub-critical class $[L^d+L^\infty]^d$, as well as critical classes such as weak $L^d$ class, Kato class, Campanato-Morrey class, Chang-Wilson-T.\,Wolff class. 
\end{abstract}

\maketitle

Let $\mathcal L^d$ be the Lebesgue measure on $\mathbb R^d$, $L^p=L^p(\mathbb R^d,\mathcal L^d)$ 
the standard (real) Lebesgue 
spaces, $C_b=C_b(\mathbb R^d)$ the space of bounded continuous functions endowed with the $\sup$-norm, 
$C_\infty \subset C_b$ the closed subspace of functions vanishing at infinity.
We denote by $\mathcal B(X,Y)$ the space of bounded linear operators between Banach spaces $X \rightarrow Y$, endowed with the operator norm $\|\cdot\|_{X \rightarrow Y}$;  $\mathcal B(X):=\mathcal B(X,X)$. Put $\|\cdot\|_{p \rightarrow q}:=\|\cdot\|_{L^p \rightarrow L^q}$.

\medskip

\textbf{1.~}Let $d \geq 3$, $b:\mathbb R^d \rightarrow \mathbb R^d$. The problem of existence and uniqueness of a  weak solution to the stochastic differential equation (SDE)
\begin{equation}
\label{sde0}
X(t) = x - \int_0^t b(X(s))ds + \sqrt{2}W(t), \quad t \geq 0, \quad x \in \mathbb R^d,
\end{equation}
with a locally unbounded vector field $b$, has been investigated by many authors. The first principal result is due to \cite{P}: if $b \in [L^p+L^\infty]^d$, $p>d$, then there exists a unique in law weak solution to \eqref{sde0}. By the results in \cite{CW}, 
a unique in law weak solution to \eqref{sde0}
exists for $b$ from the Kato class $\mathbf{K}^{d+1}_0$. 
(Recall that
a $b:\mathbb R^d \rightarrow \mathbb R^d$ belongs to the Kato class $\mathbf{K}^{d+1}_\delta$, $0<\delta<1$, if $|b| \in L^1_{\loc}$ and there exists $\lambda = \lambda_\delta > 0$ such that
$$
\| |b| (\lambda - \Delta)^{-\frac{1}{2}} \|_{1 \rightarrow 1} \leqslant \delta;
$$
and $\mathbf{K}^{d+1}_0:=\cap_{\delta>0}\mathbf{K}^{d+1}_\delta \; (\supsetneq [L^p+L^\infty]^d, p>d)$.)

\begin{definition*}
Fix $\delta \in ]0,1[.$ A $b: \mathbb{R}^d \rightarrow  \mathbb{R}^d$ belongs to $\mathbf{F}_\delta^{\scriptscriptstyle 1/2},$ the class of \textit{weakly} form-bounded vector fields, if 
$|b| \in L^1_{\loc}$ and there exists $\lambda = \lambda_\delta > 0$ such that
$$
\| |b|^\frac{1}{2} (\lambda - \Delta)^{-\frac{1}{4}} \|_{2 \rightarrow 2} \leq \sqrt{\delta}.
$$
\end{definition*}
The class $\mathbf{F}_\delta^{\scriptscriptstyle 1/2}$ contains, as proper subclasses, a sub-critical  class $ [L^d+L^\infty]^d$ ($\subsetneq \mathbf{F}^{\scriptscriptstyle 1/2}_0:=\cap_{\delta>0}\mathbf{F}_\delta^{\scriptscriptstyle 1/2}$), as well as critical classes such as the Kato class $\mathbf{K}^{d+1}_\delta$, the weak $L^d$ class, the Campanato-Morrey class, the Chang-Wilson-T.\,Wolff class, see \cite[sect.\,4]{KS2}.

\medskip

Set $m_d:=\pi^{\frac{1}{2}} (2e)^{-\frac{1}{2}} d^\frac{d}{2} (d-1)^{\frac{1-d}{2}}.$ Let $b \in \mathbf{F}^{\scriptscriptstyle 1/2}_\delta$ for some $\delta$ such that $m_d \delta < \frac{4(d-2)}{(d-1)^2}.$
 
Assume that $\{b_n\} \subset [L^\infty \cap C^1]^d \cap \mathbf{F}^{\scriptscriptstyle 1/2}_{\delta_1}$, $m_d \delta_1 < \frac{4(d-2)}{(d-1)^2}$, $b_n \rightarrow b$ strongly in $[L^1_\loc]^d$. Then \cite[Theorem 2]{Ki}, \cite[Theorem 4.4]{KS2}
\[
s \mbox{-} C_\infty \mbox{-}\lim_n e^{-t\Lambda_{C_\infty}(b_n)} 
\]
exists uniformly in $t\in [0,1],$ and hence determines a positivity preserving $L^\infty$ contraction $C_0$ semigroup $e^{-t\Lambda_{C_\infty}(b)}$ (Feller semigroup).

Here $\Lambda_{C_\infty}(b_n):=-\Delta + b_n\cdot\nabla$ of domain $(1-\Delta)^{-1}C_\infty(\mathbb R^d).$ 

 For instance, one can take
\begin{equation}
\label{b_n}
b_n:=\gamma_{\varepsilon_n} \ast \mathbf{1}_nb, \qquad n=1,2,\dots,
\end{equation}
where $\mathbf{1}_n$ is the indicator of $\{x \in \mathbb R^d:|x| \leq n, |b(x)| \leq n\}$ and
$\gamma_\varepsilon(x):=\frac{1}{\varepsilon^{d}}\gamma\left(\frac{x}{\varepsilon}\right)$ is the K.\,Friedrichs mollifier, i.e.\;$\gamma(x):=c \exp\big(\frac{1}{|x|^2-1}\big)\mathbf{1}_{|x|<1}$ with the constant $c$ adjusted to $\int_{\mathbb R^d} \gamma(x)dx=1$, for appropriate $\varepsilon_n \downarrow 0$.

\medskip

\textbf{2.\;}The space 
$D([0,\infty[,\bar{\mathbb R}^d)$ is defined to be the set of all right-continuous functions $X:[0,\infty[ \rightarrow 
\bar{\mathbb R}^d$ (here and elsewhere, $\bar{\mathbb{R}}^d:=\mathbb R^d \cup \{\infty\}$ is the one-point compactification of $\mathbb R^d$) having the left limits,  such that $X(t)=\infty$, $t>s$, whenever $X(s)=\infty$ or $X(s-)=\infty$.

By $\mathcal F_t \equiv \sigma\{X(s) \mid 0 \leq s\leq t, X \in D([0,\infty[,\bar{\mathbb R}^d)\}$ denote the minimal $\sigma$-algebra containing all cylindrical sets
$\{X \in D([0,\infty[,\bar{\mathbb R}^d)\mid \bigl(X(s_1),\dots,X(s_n)\bigr) \in A, A \subset (\bar{\mathbb R}^{d})^n \text{ is open}\}_{0 \leq s_1 \leq \dots \leq s_n \leq t}$;

By a classical result, for a given Feller semigroup $T^t$ on $C_\infty(\mathbb R^d)$, there exist probability measures $\{\mathbb P_x\}_{x \in \mathbb R^d}$ on $\mathcal F_\infty \equiv \sigma\{X(s) \mid 0 \leq s < \infty, X \in D([0,\infty[,\bar{\mathbb R}^d)\}$ such that $(D([0,\infty[,\bar{\mathbb R}^d), \mathcal F_t, \mathcal F_\infty, \mathbb P_x)$ is a Markov process (strong Markov after completing the filtration) and
$$
\mathbb E_{\mathbb P_x}[f(X(t))]=T^t f(x), \quad X \in D([0,\infty[,\bar{\mathbb R}^d), \quad f \in C_\infty, \quad x \in \mathbb R^d.
$$

The space $C([0,\infty[,\mathbb R^d)$ is defined to be the set of all continuous functions $X:[0,\infty[ \rightarrow \mathbb R^d$.

Set 
$\mathcal G_t:=\sigma\{X(s) \mid  0 \leq s \leq t, X \in C([0,\infty[,\mathbb R^d)\}$, $\mathcal G_\infty:=\sigma\{X(s) \mid  0 \leq s < \infty, X \in C([0,\infty[,\mathbb R^d)\}$.

\begin{theorem}[Main result]
\label{mainthm}
Let $d \geq 3$, $b \in \mathbf{F}^{\scriptscriptstyle 1/2}_\delta$, $m_d\delta<4\frac{d-2}{(d-1)^2}$.
Let $(D([0,\infty[,\bar{\mathbb R}^d), \mathcal F_t, \mathcal F_\infty, \mathbb P_x)$ be the Markov process
determined by $T^t=e^{-t\Lambda_{C_\infty}(b)}$. 
The following is true for every $x \in \mathbb R^d$:

\smallskip

{\rm(\textit{i})} The trajectories of the process are $\mathbb P_x$ a.s.\,finite and continuous on $0 \leq t <\infty$.

\smallskip

We denote $\mathbb P_x\upharpoonright (C([0,\infty[,\mathbb R^d),\mathcal G_\infty)$ again by $\mathbb P_x$.

\smallskip

{\rm(\textit{ii})}  $\mathbb E_{\mathbb P_x}\int_0^t |b(X(s))|ds<\infty$, $X \in C([0,\infty[,\mathbb R^d)$. 

\smallskip

{\rm(\textit{iii})}  
There exists a $d$-dimensional Brownian motion $W(t)$ on $(C([0,\infty[,\mathbb R^d),\mathcal G_t, \mathbb P_x)$ such that $\mathbb P_x$ a.s.
\begin{equation}
\label{sde1}
X(t) = x - \int_0^t b(X(s))ds + \sqrt{2}W(t), \quad t \geq 0,
\end{equation}
i.e.\,$\bigl((X(t),W(t)), (C([0,\infty[,\mathbb R^d),\mathcal G_t, \mathcal G_\infty, \mathbb P_x)\bigr)$ is a weak solution to the SDE \eqref{sde1}.
\end{theorem}

\begin{remark}
\label{rem_uniq}
One can show, using the methods of this paper, that if $\{\mathbb Q_x\}_{x \in \mathbb R^d}$ is another weak solution to \eqref{sde1} such that
$$
\mathbb Q_x=w{\mbox-}\lim_n \mathbb P_x(\tilde{b}_n) \quad \text{for every $x \in \mathbb R^d$},
$$
where $\{\tilde{b}_n\} \subset \mathbf{F}^{\scriptscriptstyle 1/2}_{\delta_1},$ $m_d\delta_1<4 \frac{d-2}{(d-1)^2},$ then $\{\mathbb Q_x\}_{x \in \mathbb R^d}=\{\mathbb P_x\}_{x \in \mathbb R^d}.$  
\end{remark}

Theorem \ref{mainthm} covers critical-order singularities of $b$, as the following example shows.

\begin{example}
Consider the vector field ($d \geq 3$) $$b(x):=  c |x|^{-2} x, \qquad c>0,$$
then $b \in \mathbf{F}^{\scriptscriptstyle 1/2}_\delta$, $c=\frac{d-2}{2}\sqrt{\delta}.$ 

1) If $c<2m_d^{-1}(d-2)^2(d-1)^{-2},$ then by Theorem \ref{mainthm}, the SDE
\[
X(t) =  - \int_0^t b(X(s))ds + \sqrt{2} W(t), \quad t \geq 0.
\]
has a weak solution. (For this particular vector field the result is, in fact, stronger, see Remark \ref{rem2} below.)

2) If $c \geq d$, then the SDE doesn't have a weak solution. 

Indeed, following \cite[Example 1.17]{CE}, suppose by contradiction that there is a weak solution to the SDE if $c \geq d$, i.e.~there are a continuous process $X(t)$ and a Brownian motion $W(t)$ on a probability space $(\Upsilon,\mathcal  F_t,\mathbb Q)$ such that $\int_0^t |b(X(s))|ds<\infty$ and the SDE holds $\mathbb Q$ a.s.
Then $X(t)=(X_1(t),\dots,X_d(t))$ is a continuous semimartingale with cross-variation $[X_i,X_k]_t=2\delta_{ik}t$.
By It\^{o}'s formula,
$$
|X(t)|^2= - 2\int_0^t X(s)\cdot b(X(s))ds + 2\sqrt{2}\int_0^t X(s)dW(s) + 2\int_0^t d[W,W]_s,
$$
i.e.~
$$
|X(t)|^2=-2  c \int_0^t \mathbf{1}_{X(s)\neq 0} ds + 2\sqrt{2}\int_0^t X(s)dW(s) + 2 t d.
$$
If we accept that $\int_0^t \mathbf{1}_{X(s)=0}ds=0$ a.s., then, clearly, 
$$
|X(t)|^2=2(d-c)  \int_0^t \mathbf{1}_{X(s)\neq 0} ds + 2\sqrt{2}\int_0^t X(s)dW(s) \quad \text{a.s.}
$$
Therefore, $|X(t)|^2 \geq 0$ is a local supermartingale if $c > d$ and is a local martingale if $c = d$.
Then a.s. $|X(0)|=0 \Rightarrow X(t)=0$, which contradicts to  $[X_1,X_1]_t=2t$.

It remains to prove that $\int_0^t \mathbf{1}_{X(s)=0}ds=0$ a.s. It suffices to show that 
$\int_0^t \mathbf{1}_{X_1(s)=0}ds=0$ a.s. Since $X_1(t)$ is a continuous semimartingale, $[X_1,X_1]_t=2t$, by the occupation times formula 
$
 \int_0^t \mathbf{1}_{X_1(s)=0} d [X_1,X_1]_s = \int_{-\infty}^\infty \mathbf{1}_{a=0}L_{X_1}^a(t)da = 0$ $\text{ a.s.},
$
where $L_{X_1}^a(t)$ is the local time of $X_1$ at $a$ on $[0,t]$.\qed
\end{example}

\begin{remark}

\label{rem2}
Recall the following

\begin{definition*}
A $b:\mathbb R^d \rightarrow \mathbb R^d$ belongs to $\mathbf{F}_\delta$, the class of form-bounded vector fields, if 
$|b| \in L^2_{\loc}$ and there exists $\lambda = \lambda_\delta > 0$ such that 
$$
\| |b| (\lambda - \Delta )^{-\frac{1}{2}} \|_{2 \rightarrow 2} \leq \sqrt{\delta}. 
$$
\end{definition*}
Note that
$\mathbf{F}_{\delta_1} \subsetneq \mathbf{F}_\delta^{\scriptscriptstyle 1/2}$
for  
$\delta = \sqrt{\delta_1}.$

\smallskip

For $b \in \mathbf{F}_{\delta_1}$, the constraint $m_d\sqrt{\delta_1}<4\frac{d-2}{(d-1)^2}$ in Theorem \ref{mainthm} can be relaxed to $\delta_1<1 \wedge \big(\frac{2}{d-2}\big)^2$. The proof of Theorem \ref{mainthm} extends to such $b$  after replacing Lemma A
below by evident modifications of \cite[Lemma 5]{KS}, \cite[Theorem 3.7]{KS2}. 

For $b(x):=c|x|^{-2}x \in \mathbf{F}_{\delta_1}$, $\delta_1:=c^2\frac{4}{(d-2)^2}$, the result is even stronger: $-1<c<\frac{1}{2}$ if $d=3$, $-\infty<c<1$ if $d=4$, $-\infty<c<(d-3)/2$ if $d \geq 5$
(after replacing Lemma A by evident modifications of Theorems 3.8, 3.9 in \cite{KS2}).

We refer to \cite{KS2} for a more detailed discussion on classes $\mathbf{F}_{\delta_1}$, $\mathbf{F}_\delta^{\scriptscriptstyle 1/2}$.
\end{remark}

\section{Preliminaries}

\label{sect_1}

Denote by $C^{0,\alpha}=C^{0,\alpha}(\mathbb R^d)$ the space of H\"{o}lder continuous functions ($0<\alpha<1$), $\mathcal S$ the L.~Schwartz space of test functions, $W^{k,p}=W^{k,p}(\mathbb R^d,\mathcal L^d)$ ($k=1,2$)
the standard 
Sobolev spaces, $\mathcal W^{\alpha,p}$, $\alpha>0$, the Bessel potential space endowed with norm $\|u\|_{p,\alpha}:=\|g\|_p$,  
$u=(1-\Delta)^{-\frac{\alpha}{2}}g$, $g \in L^p$, and $\mathcal W^{-\alpha,p'}$, $p'=p/(p-1)$, the anti-dual of $\mathcal W^{\alpha,p}$.

\smallskip

The proof of Theorem \ref{mainthm} is based on the following analytic results \cite[Theorems 1, 2]{Ki}, \cite[Theorems 4.3, 4.4]{KS2}.

Set $$I_s:=\bigl]\frac{2}{1+\sqrt{1-m_d\delta}},\frac{2}{1-\sqrt{1-m_d\delta}} \bigr[.$$
For every $p \in I_s$, there exists a holomorphic semigroup $e^{-t\Lambda_p(b)}$ on $L^p$ such that the resolvent set of $-\Lambda_p(b)$ contains the half-plane $\mathcal O:=\{\zeta \in \mathbb C:\Real \zeta \geq \kappa_d\lambda\}$,
\begin{align}
\label{theta_repr}
(\zeta+\Lambda_p(b))^{-1}= (\zeta - \Delta)^{-1} - (\zeta - \Delta)^{-\frac{1}{2}-\frac{1}{2q}} Q_{p}(q) (1 + T_p)^{-1} G_{p}(r) (\zeta - \Delta)^{-\frac{1}{2r'}}, \quad \zeta \in \mathcal O,
\end{align}
where $1 \leq r<p<q$, $\kappa_d:=\frac{d}{d-1}$, $Q_p(q), G_p(r), T_p \in \mathcal B(L^p)$, 
\begin{align*}
G_{p}(r):= b^\frac{1}{p} \cdot \nabla (\zeta - \Delta)^{-\frac{1}{2} - \frac{1}{2r}},  \quad b^\frac{1}{p}:=|b|^{\frac{1}{p}-1}b,
\end{align*}
$Q_p(q)$, $T(p)$ are the extensions by continuity of densely defined on $\mathcal{E} :=  \bigcup_{\epsilon >0} e^{-\epsilon|b|} L^p$ operators
\begin{align*}
Q_p(q) \upharpoonright \mathcal E: = (\zeta -\Delta)^{- \frac{1}{2q'}} |b|^{\frac{1}{p'}}, \quad
T_p \upharpoonright \mathcal E:= b^\frac{1}{p} \cdot \nabla (\zeta -\Delta)^{-1} |b|^\frac{1}{p^\prime},
\end{align*}
$$
\|T_p\|_{p \rightarrow p} \leq m_dc_p\delta, \quad  c_p:= \frac{pp'}{4}, \quad m_dc_p\delta<1 \;\; (\Leftrightarrow p \in I_s).
$$
\begin{equation*}
e^{-t\Lambda_p(b)}=s \mbox{-}L^p \mbox{-} \lim_n e^{-t\Lambda_p(b_n)} \quad \text{(uniformly on every compact interval of $t \geq 0$) },
\end{equation*} 
where $b_n$'s are given by \eqref{b_n}, $\Lambda_p(b_n) := -\Delta + b_n \cdot \nabla,$  $D(\Lambda_p(b_n)) = W^{2,p}$. 

By \eqref{theta_repr}, $$(\zeta + \Lambda_p(b))^{-1} \in \mathcal B(\mathcal W^{-\frac{1}{r'},p}, \mathcal W^{1+\frac{1}{q},p}).$$

Fix numbers $p \in I_s$, $p>d-1$ \footnote{Since $m_d\delta<4\frac{d-2}{(d-1)^2}$, such $p$ exists.} and $q$ sufficiently close to $p$. By \eqref{theta_repr} and the Sobolev Embedding Theorem, $(\zeta+\Lambda_{p}(b))^{-1}[L^p] \subset C^{0,\alpha}$, $\alpha<1-\frac{d-1}{p}$. 
Define $\Lambda_{C_\infty}(b)$ by
$$
(\mu+\Lambda_{C_\infty}(b))^{-1}:=\bigl((\mu+\Lambda_{p}(b))^{-1} \upharpoonright L^p \cap C_\infty \bigr)^{\clos}_{C_\infty \rightarrow C_\infty}, \quad \mu \geq \kappa_d\lambda.
$$
Then
\begin{equation}
\label{SF}
 \bigl(e^{-t\Lambda_{C_\infty}(b)} \upharpoonright L^p \cap C_\infty \bigr)^{\cl}_{L^p \rightarrow C_\infty} \in \mathcal B(L^p,C_\infty), \qquad p \in \bigl]d-1,\frac{2}{1-\sqrt{1-m_d\delta}} \bigr[, \; t>0.
\end{equation}
\begin{equation}
\label{conv_c}
e^{-t\Lambda_{C_\infty}(b)}=s \mbox{-}C_\infty \mbox{-} \lim_n e^{-t\Lambda_{C_\infty}(b_n)} \quad \text{(uniformly on every compact interval of $t \geq 0$)},
\end{equation}
where $D(\Lambda_{C_\infty}(b_n))=(1-\Delta)^{-1}C_\infty$.

%By construction, the semigroups $e^{-t\Lambda_{C_\infty}(b)}$, $e^{-t\Lambda_{q}(b)}$ are consistent:
%\begin{equation}
%\label{CP}
%e^{-t\Lambda_{C_\infty}(b)} \upharpoonright L^p \cap C_\infty = e^{-t\Lambda_{p}(b)} \upharpoonright L^p \cap C_\infty \quad \text{(after a change on a set of measure zero)}, \quad t \geq 0.
%\end{equation}

The following estimates are direct consequences of \eqref{theta_repr}: There exist constants $C_i=C_i(\delta,p)$, $i=1,2$, such that, for all $h \in C_c$ and $\mu \geq \kappa_d\lambda_\delta,$
\begin{equation}
\label{j_2}
\bigl\|(\mu+\Lambda_{C_\infty}(b))^{-1}|b_m| h \bigr\|_{\infty} \leq C_1\||b_m|^{\frac{1}{p}}h\|_p,  
\end{equation}
\begin{equation}
\label{rem_j3}
\|(\mu+\Lambda_{C_\infty}(b))^{-1}|b_m-b_n|h\|_\infty  \leq C_2 \bigl\| |b_m-b_n|^{\frac{1}{p}} h\bigr\|_p.
\end{equation}

\bigskip

Our proof of Theorem \ref{mainthm} employs also the following weighted estimates.
Set $$\rho(y)\equiv\rho_l (y):=(1+l |y|^2)^{-\nu},   \;\;\; \nu > \frac{d}{2p}+1, \;\; l>0, \;\; y \in \mathbb R^d.$$

\begin{lemmaA}
\label{prop_apr_w_}
Fix $p \in I_s$, $p>d-1$. There exist constants $K_i=K_i(\delta,p)$, $i=1,2$ such that, for all $h \in C_c(\mathbb R^d)$, $\mu \geq \kappa_d\lambda_\delta$ and all sufficiently small $l=l(\delta,p)>0$,
\begin{equation}
\label{j_1_w}
\tag{$E_1$}
\bigl\|\rho (\mu+\Lambda_{C_\infty}(b_n))^{-1} h \bigr\|_{\infty} \leq K_1\|\rho h\|_p,
\end{equation}
\begin{equation}
\label{j_2_w}
\tag{$E_2$}
\bigl\|\rho (\mu+\Lambda_{C_\infty}(b_n))^{-1}|b_m| h \bigr\|_{\infty} \leq K_2\| |b_m|^\frac{1}{p} \rho h\|_p. 
\end{equation}
\end{lemmaA}
This technical lemma is proven in the appendix.

\medskip

\section{Proof of Theorem \ref{mainthm}}

%\textbf{B.~}Let us start the proof of Theorem \ref{mainthm}. 

\begin{lemma}
\label{ae_rem}
For every $x \in \mathbb R^d$ and $t>0$, $b_n(X(t)) \rightarrow b(X(t))$ $\mathbb P_x$ a.s. as $n \uparrow \infty$.
\end{lemma}
\begin{proof}
By \eqref{SF} and the Dominated Convergence Theorem, for any $\mathcal L^d$-measure zero set $G \subset \mathbb R^d$ and every $t>0$, $\mathbb P_x[X(t) \in G]=0$. 
Since $b_n \rightarrow b$ pointwise in $\mathbb R^d$ outside of an $\mathcal L^d$-measure zero set, we have the required.
\end{proof}

Let $\mathbb P^n_x$  be the probability measures associated with $e^{-t\Lambda_{C_\infty}(b_n)}$, $n=1,2,\dots$
Set $\mathbb E_x:=\mathbb E_{\mathbb P_x}$, and $\mathbb E^n_x:=\mathbb E_{\mathbb P_x^n}$. 

Fix a $\upsilon \in C^\infty([0,\infty[)$, $\upsilon (s)=1$ if $0 \leq s \leq 1$, $\upsilon (s)=0$ if $s \geq 2$.  
Set
\begin{equation}
\label{xi_k}
\xi_k(y):=\left\{
\begin{array}{ll}
\upsilon (|y|+1-k) & |y| \geq k, \\
1 & |y|<k.
\end{array}
\right.
\end{equation}

\begin{lemma}
\label{finite_prop} 
For every $x \in \mathbb R^d$ and $t>0$, $\mathbb P_x[X(t)=\infty]=0$.
\end{lemma}

\begin{proof}
First, let us show that for every $\mu \geq \kappa_d\lambda_\delta$, 
\begin{equation}
\label{conv_n9}
\int_0^\infty e^{-\mu t}\mathbb E^n_x[\xi_k(X(t))]dt \rightarrow \frac{1}{\mu} \quad \text{ as $k \uparrow \infty$ \textit{uniformly in} $n$}.
\end{equation}
Since $\int_0^\infty e^{-\mu t}\mathbb E^n_x[\mathbf{1}_{\mathbb R^d}(X(t))]dt=\frac{1}{\mu}$, \eqref{conv_n9} is equivalent to
$\int_0^\infty e^{-\mu t}\mathbb E^n_x[(\mathbf{1}_{\mathbb R^d}-\xi_k)(X(t))]dt  \rightarrow 0$ as $k \uparrow \infty$ uniformly in $n$. 
We have
\begin{align*}
& \int_0^\infty e^{-\mu t}\mathbb E^n_x[(\mathbf{1}_{\mathbb R^d}-\xi_k)(X(t))]dt \\
& \text{(we use the Dominated Convergence Theorem)} \\
& = \lim_{r \uparrow \infty}\int_0^\infty e^{-\mu t}\mathbb E^n_x[\xi_r(1-\xi_k)(X(t))]dt \\
& = \lim_{r \uparrow \infty}(\mu+\Lambda_{C_\infty}(b_n))^{-1}[\xi_r(1-\xi_k)](x) \\
& \text{(we apply crucially \eqref{j_1_w})} \\
& \leq \rho(x)^{-1} K_1\lim_{r \uparrow \infty} \|\rho\xi_r (1-\xi_k)\|_p \leq \rho(x)^{-1} K_1\|\rho (1-\xi_k) \|_p \rightarrow 0 \quad \text{as $k \uparrow \infty$},
\end{align*}
which yields \eqref{conv_n9}.

Now, since  $\mathbb E_x[\xi_k(X(t))]=\lim_{n} \mathbb E^n_x[\xi_k(X(t))]$ uniformly on every compact interval of $t \geq 0$, see \eqref{conv_c}, it follows from \eqref{conv_n9} that $$\int_0^\infty e^{-\mu t}\mathbb E_x[\xi_k(X(t))]dt \rightarrow \frac{1}{\mu} \quad \text{ as }k \uparrow \infty.$$

Finally, suppose that  $\mathbb P_x[X(t)=\infty]$ is strictly positive for some $t>0$.
By the construction of $\mathbb P_x$, $t \mapsto \mathbb P_x[X(t)=\infty]$ is non-decreasing, and so
%$\varkappa:=\int_0^\infty e^{-\mu t}\mathbb P_x[X(t)=\infty]dt>0$. Now,
$\varkappa:=\int_0^\infty e^{-\mu t}\mathbb E_x[\mathbf{1}_{X(t)=\infty}]dt>0$. Now,
$$
\frac{1}{\mu}=\int_0^\infty e^{-\mu t}\mathbb E_x[\mathbf{1}_{\bar{\mathbb R}^d}(X(t))]dt \geq \varkappa + \int_0^\infty e^{-\mu t}\mathbb E_x[\xi_k(X(t))]dt.
$$
Selecting $k$ sufficiently large, we arrive at contradiction. 
\end{proof}

The space $D([0,\infty[,\mathbb R^d)$ is defined to be the subspace of $D([0,\infty[,\bar{\mathbb R}^d)$ consisting of the trajectories $X(t) \neq \infty$, $0 \leq t <\infty$. 
Let $\mathcal F'_{t}:=\sigma(X(s) \mid 0 \leq s \leq t, X \in D([0,\infty[,\mathbb R^d))$, $\mathcal F'_{\infty}:=\sigma(X(s) \mid 0 \leq s <\infty, X \in D([0,\infty[,\mathbb R^d))$.

By Lemma \ref{finite_prop}, $(D([0,\infty[,\mathbb R^d),\mathcal  F_\infty')$ has full $\mathbb P_x$-measure in $(D([0,\infty[,\bar{\mathbb R}^d),\mathcal  F_\infty)$.
We denote the restriction of $\mathbb P_x$ from $(D([0,\infty[,\bar{\mathbb R}^d),\mathcal  F_\infty)$ to $(D([0,\infty[,\mathbb R^d),\mathcal  F_\infty')$ again by $\mathbb P_x$.

\begin{lemma}
\label{Y_prop}
For every $x \in \mathbb R^d$ and $g \in C_c^\infty(\mathbb R^d)$, 
$$
g(X(t)) - g(x) + \int_0^t (-\Delta g + b\cdot\nabla g)(X(s))ds, $$
is a martingale relative to $(D([0,\infty[,\mathbb R^d),\mathcal F'_t, \mathbb P_x)$.

\end{lemma}

\begin{proof}
Fix $\mu \geq \kappa_d\lambda_\delta$. In what follows, $0<t\leq T<\infty$.

\smallskip

$(\mathbf a)$ $\mathbb E_x \int_0^t \bigl|b\cdot\nabla g \bigr|(X(s))ds<\infty$. Indeed, 
\begin{align*}
& \notag \mathbb E_x \int_0^t \bigl|b\cdot\nabla g \bigr|(X(s))ds  \\
& \notag (\text{we apply Fatou's Lemma, cf.\,Lemma \ref{ae_rem}}) \\
& \notag \leq \liminf_n \mathbb E_x \int_0^t \bigl|b_n\cdot\nabla g \bigr|(X(s))ds  = \liminf_n  \int_0^t e^{-s\Lambda_{C_\infty}(b)}\bigl|b_n\cdot\nabla g \bigr|(x)ds   \\
& =\liminf_n  \int_0^t e^{\mu s} e^{-\mu s }e^{-s\Lambda_{C_\infty}(b)}\bigl|b_n\cdot\nabla g \bigr|(x)ds  \\
& \notag \leq e^{\mu T} \liminf_n (\mu+\Lambda_{C_\infty}(b))^{-1}|b_n| |\nabla g|(x) \\
& \notag \text{(we apply \eqref{j_2} with } h=|\nabla g|) \\
& \notag \leq C_1 e^{\mu T}\liminf_n  \langle |b_n| |\nabla g|^p \rangle^{\frac{1}{p}} \leq C_1 e^{\mu T} 2^\frac{1}{p}\big(\langle |b| |\nabla g|^p \rangle^\frac{1}{p}+ \lim_n \langle |b-b_n| |\nabla g|^p \rangle ^\frac{1}{p}\big) \\
& = C_1 e^{\mu T}2^\frac{1}{p} \langle |b| |\nabla g|^p \rangle^\frac{1}{p} <\infty.
\end{align*}

$(\mathbf b)$ 
$
\mathbb E^n_x[g(X(t))] \rightarrow \mathbb E_x[g(X(t))], 
$
$
\mathbb E_x^n \int_0^t ( b_n\cdot\nabla g)(X(s))ds \rightarrow \mathbb E_x\int_0^t (b\cdot\nabla g)(X(s))ds, 
$
and also, for $h \in C_c^\infty$,
$
\mathbb E^n_x\int_0^t
(|b_n| h)(X(s))ds \rightarrow \mathbb E_x\int_0^t
(|b|h)(X(s))ds
$
as $n \uparrow \infty$.
Indeed, the first  convergence follows from \eqref{conv_c}, the second one follows from $(\mathbf a)$, and the third one from $\mathbb E_x \int_0^t (|b||h|)(X(s))ds<\infty$, a straightforward modification of $(\mathbf a)$.

$(\mathbf c)$ $\mathbb E_x\int_0^t (b_n\cdot\nabla g)(X(s))ds - \mathbb E^n_x\int_0^t (b_n\cdot\nabla g)(X(s))ds\rightarrow 0.$ We have:
\begin{align*}
& \mathbb E_x\int_0^t (b_n\cdot\nabla g)(X(s))ds - \mathbb E^n_x\int_0^t (b_n\cdot\nabla g)(X(s))ds \\
& = \int_0^t \left( e^{-s\Lambda_{C_\infty}(b)} - e^{-s\Lambda_{C_\infty}(b_n)} \right) (b_n\cdot\nabla g)(x)ds \\
& =\int_0^t \left( e^{-s\Lambda_{C_\infty}(b)} - e^{-s\Lambda_{C_\infty}(b_n)} \right) ((b_n-b_m)\cdot\nabla g)(x)ds  \\
& + \int_0^t \left( e^{-s\Lambda_{C_\infty}(b)} - e^{-s\Lambda_{C_\infty}(b_n)} \right) (b_m\cdot\nabla g)(x)ds =: S_1 + S_2,
\end{align*}
where $m$ is to be chosen. 
Arguing as in the proof of $(\mathbf a)$, we obtain:
\begin{align*}
S_1(x) & \leq   
e^{\mu T} (\mu+\Lambda_{C_\infty}(b))^{-1}|(b_n-b_m)\cdot\nabla g|(x)  + e^{\mu T} (\mu+\Lambda_{C_\infty}(b_n))^{-1}|(b_n-b_m)\cdot\nabla g|(x).
\end{align*}
Since $b_n-b_m \rightarrow 0$ in $L^1_{\loc}$ as $n,m \uparrow \infty$, \eqref{rem_j3} yields $S_1 \rightarrow 0$ as $n,m \uparrow \infty$.
Now, fix a sufficiently large $m$. Since $e^{-s\Lambda_{C_\infty}(b)}=s\text{-}C_\infty\text{-}\lim_n e^{-s\Lambda_{C_\infty}(b_n)}$ uniformly in $0 \leq s \leq T$, cf.~\eqref{conv_c}, 
we have $S_2 \rightarrow 0$ as $n \uparrow \infty$. 
The proof of $(\mathbf c)$ is completed.

\medskip

Now we are in position to complete the proof of Lemma \ref{Y_prop}.
Since $b_n \in [C_c^\infty(\mathbb R^d)]^d$, 
$$g(X(t)) - g(x) + \int_0^t (-\Delta g + b_n\cdot\nabla g)(X(s))ds \text{ is a martingale under $\mathbb P^n_x$,}
$$
so the function
$$
x \mapsto \mathbb E^n_x[g(X(t))] - g(x) +\mathbb E^n_x\int_0^t (-\Delta g + b_n\cdot\nabla g)(X(s))ds \quad \text{ is identically zero in } \mathbb R^d.
$$
Thus by $(\mathbf b)$, the function
$$
x \mapsto \mathbb E_x[g(X(t))] - g(x) +\mathbb E_x\int_0^t (-\Delta g + b\cdot\nabla g)(X(s))ds \quad  \text{ is identically zero in } \mathbb R^d,
$$
i.e. $g(X(t)) - g(x) +\int_0^t (-\Delta g + b\cdot\nabla g)(X(s))ds$ is a martingale under $\mathbb P_x$.
\end{proof}

\begin{lemma}
\label{cont_prop}
For $x \in \mathbb R^d$, 
$C([0,\infty[,\mathbb R^d)$ has full $\mathbb P_x$-measure in $D([0,\infty[,\mathbb R^d)$.

\end{lemma}

\begin{proof} Let $A$, $B$ be arbitrarily bounded closed sets in $\mathbb R^d$, $\dist(A,B)>0$. 
Fix $g \in C_c^\infty(\mathbb R^d)$ such that $g = 0$ on $A$, $g = 1$ on $B$. Set ($X \in D([0,\infty[,\mathbb R^d)$)
$$
M^g(t):=g(X(t)) - g(x) + \int_0^t (-\Delta g + b\cdot\nabla g)(X(s))ds, \quad
K^g(t):=\int_0^t \mathbf{1}_A(X(s-))dM^g(s),
$$
then
\begin{align*}
K^g(t) &=\sum_{s \leq t} \mathbf{1}_A \left(X(s-)\right)g(X(s)) +
\int_0^t \mathbf{1}_A(X(s-))\bigl(-\Delta g + b\cdot\nabla g \bigr)(X(s))ds \\
& =\sum_{s \leq t} \mathbf{1}_A \left(X(s-)\right)g(X(s)).
\end{align*}
By Lemma \ref{Y_prop}, $M^g(t)$ is a martingale, and hence so is $K^g(t)$. Thus, $\mathbb{E}_x\bigl[\sum_{s \leq t} \mathbf{1}_A (X(s-))g(X(s))\bigr]=0.$ Using the Dominated Convergence Theorem, we obtain
$\mathbb{E}_x\bigl[\sum_{s \leq t} \mathbf{1}_A (X(s-))\mathbf{1}_B(X(s))\bigr]=0$. The proof of Lemma \ref{cont_prop} is completed.
\end{proof}

We denote the restriction of $\mathbb P_x$ from $(D([0,\infty[,\mathbb R^d), \mathcal F_\infty')$  to $(C([0,\infty[,\mathbb R^d),\mathcal G_\infty)$ again by $\mathbb P_x$. Lemma \ref{Y_prop} and Lemma \ref{cont_prop} combined yield

\begin{lemma}
\label{thm1}
For every $x \in \mathbb R^d$ and $g \in C_c^\infty(\mathbb R^d)$, 
$$
g(X(t)) - g(x) + \int_0^t (-\Delta g + b\cdot\nabla g)(X(s))ds, \quad X \in C([0,\infty),\mathbb R^d),$$
is a continuous martingale relative to $(C([0,\infty[,\mathbb R^d),\mathcal G_t, \mathbb P_x)$.
\end{lemma}

\begin{lemma}
\label{thm2}
For every $x \in \mathbb R^d$ and  $t>0$,
$
\mathbb E_x\int_0^t |b(X(s))|ds<\infty,
$ and, for $f(y)=y_i$ or $f(y)=y_iy_j$, $1 \leq i,j \leq d$,
$$
f(X(t)) - f(x) + \int_0^t (-\Delta f + b\cdot\nabla f)(X(s))ds, \quad X \in C([0,\infty[,\mathbb R^d),
$$
is a continuous martingale relative to $(C([0,\infty[,\mathbb R^d),\mathcal G_t, \mathbb P_x)$.
\end{lemma}

\begin{proof}

Define $f_k:=\xi_k f \in C_c^\infty(\mathbb R^d)$ (see \eqref{xi_k} for the definition of $\xi_k$).
Set $\alpha:=\|\nabla \xi_k\|_\infty$, $\beta:=\|\Delta \xi_k\|_\infty$ ($\alpha, \beta$ don't depend on $k$).
Fix $0<T<\infty$. In what follows, $0 < t \leq T$.

$(\mathbf a) \quad \mathbb E_x\int_0^t (|b| (|\nabla f| + \alpha|f|)) (X(s))ds <\infty.$

 Indeed, set $\varphi:=|\nabla f| + \alpha|f| \in C \cap W^{1,2}_{\loc}$, $\varphi_k:=\xi_{k+1} \varphi \in C_c \cap W^{1,2}$.
First, let us prove that
\[
\mathbb E^n_x\int_0^t (|b_n| \varphi_k) (X(s))ds  \leq \const \text{ independent of } n, k.
\]
Fix $p \in \bigl]d-1,\frac{2}{1-\sqrt{1-m_d\delta}} \bigr[$. Then $\sqrt{(\rho \varphi)^p} \in W^{1,2}$ (recall that $\rho(x):=(1+l|x|^2)^{-\nu}$, $\nu>\frac{d}{2p}+1.)$ We have
\begin{align*}
& \mathbb E^n_x\int_0^t (|b_n| \varphi_k) (X(s))ds=\int_0^t e^{-s\Lambda_{C_\infty}(b_n)}|b_n| \varphi_k(x)ds \\
& \leq e^{\mu T} (\mu+\Lambda_{C_\infty}(b_n))^{-1}|b_n| \varphi_k(x) \\
& (\text{we apply \eqref{j_2_w}} ) \\
& \leq e^{\mu T} \rho(x)^{-1}K_2 \langle |b_n| (\rho \varphi_k)^p \rangle^{\frac{1}{p}}   \leq e^{\mu T} \rho(x)^{-1}K_2\langle |b_n| (\rho \varphi)^p \rangle^{\frac{1}{p}}  \\
& \bigg(\text{we use } b_n \in \mathbf F^{\scriptscriptstyle 1/2}_{\delta_1}, \; m_d \delta_1 < 4 \frac{d-1}{(d-2)^2}\bigg) \\
& \leq e^{\mu T} \rho(x)^{-1}K_2 \delta_1^{\frac{1}{p}} \|(\lambda-\Delta)^\frac{1}{4} \sqrt{(\rho \varphi)^p} \|_2^{\frac{2}{p}}
 <\infty.
\end{align*}
By step $(\mathbf b)$ in the proof of Lemma \ref{Y_prop}, $
\mathbb E^n_x\int_0^t
(|b_n| \varphi_k)(X(s))ds \rightarrow \mathbb E_x\int_0^t
(|b|\varphi_k)(X(s))ds$ as $n \uparrow \infty$.
Therefore, $\mathbb E^n_x\int_0^t (|b_n| \varphi_k) (X(s))ds  \leq C$ implies
$
\mathbb E_x\int_0^t (|b| \varphi_k) (X(s))ds  \leq C \; (C \neq C(k)).
$
Now, Fatou's Lemma yields the required.

\smallskip

$(\mathbf b)$ \quad  For every $t>0$, \; 
$
\mathbb E_x\int_0^t (|\Delta f| + 2 \alpha |\nabla f| + \beta |f|)(X(t)) ds <\infty.
$

The proof is similar to the proof of $(\mathbf a)$ (use \eqref{j_1_w} instead of \eqref{j_2_w}). 

\smallskip

$(\mathbf c)$ \quad For every $t>0$,\;
$
\mathbb E_x[|f|(X(t))]<\infty. 
$

Indeed, set $g(y):=1+|y|^2$, $y \in \mathbb R^d$. Since $|f| \leq g$, it suffices to show that $\mathbb E_x[g(X(t))]<\infty$. Set $g_k(y):=\xi_k(y) g(y)$. By Lemma \ref{thm1},
$$
\mathbb E_x[g_k(X(t))]  = g_k(x) - \mathbb E_x \int_0^t  (-\Delta g_k)(X(s))ds -  \mathbb E_x \int_0^t (b \cdot \nabla g_k)(X(s))ds. 
$$
Note that
$$
\sup_k\mathbb E_x\int_0^t (|b| |g_k|) (X(s))ds <\infty, \quad \sup_k\mathbb E_x\int_0^t |\Delta g_k|(X(s)) ds <\infty
$$
for, arguing as in the proofs of $(\mathbf a)$ and $(\mathbf b)$, we have:
$$
\mathbb E_x\int_0^t (|b| (|\nabla g| + \alpha|g|)) (X(s))ds <\infty, \quad \mathbb E_x\int_0^t (|\Delta g| + 2 \alpha |\nabla g| + \beta |g|)(X(t)) ds <\infty.
$$
Therefore, $\sup_k \mathbb E_x[g_k(X(t))]<\infty$, and so, by the Monotone Convergence Theorem, $\mathbb E_x[g(X(t))]<\infty$. This completes the proof of $(\mathbf{c})$.

\smallskip

Let us complete the proof of Lemma \ref{thm2}. 
By $(\mathbf a)$, $
\mathbb E_x\int_0^t |b(X(s))|ds<\infty$.
By $(\mathbf a)$-$(\mathbf c)$,
$$
M^f(t):=f(X(t)) - f(x) + \int_0^t (-\Delta f + b\cdot\nabla f)(X(s))ds, \quad t>0,
$$
satisfies $\mathbb E_x[|M^{f}(t)|]<\infty$ for all $t>0$.
By Lemma \ref{thm1}, for every $k$, $M^{f_k}(t)$ is a martingale relative to $(C([0,\infty[,\mathbb R^d),\mathcal G_t, \mathbb P_x)$.  By $(\mathbf a)$ and the Dominated Convergence Theorem, since $|\nabla f_k| \leq |\nabla f| + \alpha|f|$ for all $k$, we have
$
\mathbb E_x\int_0^t (b \cdot \nabla f_k) (X(s))ds \rightarrow \mathbb E_x\int_0^t (b \cdot \nabla f) (X(s))ds.
$
By $(\mathbf b)$,
$
\mathbb E_x\int_0^t (-\Delta f_k) (X(s))ds \rightarrow \mathbb E_x\int_0^t (-\Delta f) (X(s))ds.
$
By $(\mathbf c)$,
$
\mathbb E_x[f_k(X(t))] \rightarrow \mathbb E_x[f(X(t))]. 
$
So,
$M^{f}(t)$ is also a martingale on $(C([0,\infty[,\mathbb R^d),\mathcal G_t, \mathbb P_x)$.
The proof of Lemma \ref{thm2} is completed.
\end{proof}

We are in position to complete the proof of Theorem \ref{mainthm}. Lemma \ref{cont_prop} yields (\textit{i}). Lemma \ref{thm2} yields (\textit{ii}). 
By classical results, Lemma \ref{thm2} yields existence of a $d$-dimensional Brownian motion $W(t)$ on $(C([0,\infty[,\mathbb R^d),\mathcal G_t, \mathbb P_x)$ such that
$
X(t) = x - \int_0^t b(X(s))ds + \sqrt{2}W(t)$, $0 \leq t<\infty$, $\text{$\mathbb P_x$ a.s.}
$
$\Rightarrow$ (\textit{iii}). The proof of Theorem \ref{mainthm} is completed.

\section*{Appendix: Proof of Lemma A}

\label{prop1_sect}

The proofs of \eqref{j_1_w} and \eqref{j_2_w} are similar. For instance, let us prove \eqref{j_1_w}.

\smallskip

We will use the bounds:
\begin{equation}
\label{est_prop21}
\|(\mu-\Delta)^{-\frac{1}{2}}|b|^{\frac{1}{p'}}\|_{p \rightarrow p} \leq C_{p,\delta}<\infty, \quad  \||b|^{\frac{1}{p}}(\mu-\Delta)^{-\frac{1}{2}}\|_{p \rightarrow p} \leq C_{p',\delta}<\infty \;\;\; (\text{by duality})
\end{equation}
(for $\|Q_p(q)\|_{p \rightarrow p} \leq C_{p,q,\delta}<\infty$, see section \ref{sect_1}).
%see \cite[Prop.\,2.1]{Ki}, \cite[Theorem 4.3]{KS2} (step (\textbf{a}) in the proof); by duality, $\|(\mu-\Delta)^{-\frac{1}{2}}|b|^{\frac{1}{p'}}\|_{p \rightarrow p} \leq C_{p',\delta}<\infty$.

\medskip

By the definition of $\rho$,
\begin{equation}
\label{eta_two_est}
\tag{$\star$}
|\nabla \rho| \leq \nu \sqrt{l}\rho \equiv C_1 \sqrt{l}\rho, \quad |\Delta \rho| \leq 2\nu (2\nu + d+2 ) l \rho \equiv C_2 l \rho.
\end{equation}
Set $u=(\mu-\Delta)^{-1}f, \; f \in C_c(\mathbb R^d)$. We have
$
(\mu-\Delta) \rho u = -(\Delta \rho)u - 2\nabla \rho \cdot \nabla u + \rho (\mu-\Delta)u,
$
and so
$$
\rho u = -(\mu-\Delta)^{-1}(\Delta \rho)u - 2(\mu-\Delta)^{-1}\nabla \rho \cdot \nabla u + (\mu-\Delta)^{-1}\rho (\mu-\Delta)u.
$$
 Thus, 
\begin{align}
\rho (\mu-\Delta)^{-1}f  = & -(\mu-\Delta)^{-1}(\Delta \rho) (\mu-\Delta)^{-1}f \tag{$\star\star$} \label{est_1}\\
&- 2(\mu-\Delta)^{-1}\nabla \rho \cdot \nabla  (\mu-\Delta)^{-1}f \notag \\
&+ (\mu-\Delta)^{-1}\rho f \notag.
\end{align}
We obtain from \eqref{est_1}:
\begin{align*}
\rho \nabla (\mu-\Delta)^{-1}f = & -(\nabla \rho) (\mu-\Delta)^{-1}f  \\
&-\nabla(\mu-\Delta)^{-1}(\Delta \rho) (\mu-\Delta)^{-1}f\\
&- 2\nabla(\mu-\Delta)^{-1}\nabla \rho \cdot \nabla  (\mu-\Delta)^{-1}f  \\
& + \nabla(\mu-\Delta)^{-1}\rho f.
\end{align*}
Then
\begin{align*}
I_0 & :=\|\rho(|b_n|^{\frac{1}{p}}+1) \nabla (\mu-\Delta)^{-1}f\|_p \\
& \leq C_1\sqrt{l}\|(|b_n|^{\frac{1}{p}}+1)\rho (\mu-\Delta)^{-1}f\|_p \\
& + C_2 l m_d\|(|b_n|^{\frac{1}{p}}+1)(\kappa_d^{-1}\mu-\Delta)^{-\frac{1}{2}}\rho |(\mu-\Delta)^{-1}f|\|_p \\
& + 2C_1 \sqrt{l} m_d \|(|b_n|^{\frac{1}{p}}+1)(\kappa_d^{-1}\mu-\Delta)^{-\frac{1}{2}}\rho |\nabla (\mu-\Delta)^{-1}f|\|_p \\
& + \|(|b_n|^{\frac{1}{p}}+1) \nabla (\mu-\Delta)^{-1} \rho f\|_p \\
& =: C_1\sqrt{l}I_1 + C_2 l m_d I_2 + 2C_1 \sqrt{l} m_d I_3 + \|(|b_n|^{\frac{1}{p}}+1) \nabla (\mu-\Delta)^{-1} \rho f\|_p.
\end{align*}

We have:
\begin{align*}
& I_3 \leq \|(|b_n|^{\frac{1}{p}}+1)(\kappa_d^{-1}\mu-\Delta)^{-\frac{1}{2}}\|_{p \rightarrow p}\|\rho \nabla (\mu-\Delta)^{-1}f\|_p \\
& (\text{we use \eqref{est_prop21}}) \\
& \leq c \|\rho \nabla (\mu-\Delta)^{-1}f\|_p  \leq c I_0.
\end{align*}
We estimate $I_1$ using \eqref{est_1} and \eqref{eta_two_est}:
\begin{align*}
& I_1 \leq C_2 l \|(|b_n|^{\frac{1}{p}}+1)(\mu-\Delta)^{-1}\rho(\mu-\Delta)^{-1}f\|_p \\
& + 2C_1\sqrt{l} \|(|b_n|^{\frac{1}{p}}+1)(\mu-\Delta)^{-1}\|_{p \rightarrow p} \|\rho \nabla (\mu-\Delta)^{-1}f\|_p \\
& + \|(|b_n|^{\frac{1}{p}}+1)(\mu-\Delta)^{-1}\rho f\|_p,
\end{align*}
and so
$
I_1 \leq C_2 l I_1 + 2C_1\sqrt{l}c I_3 + \|(|b_n|^{\frac{1}{p}}+1)(\mu-\Delta)^{-1}\rho f\|_p.
$

We estimate $I_2$ again using \eqref{eta_two_est} and \eqref{est_1}:
\begin{align*}
& I_2 \leq C_2l \|(|b_n|^{\frac{1}{p}}+1)(\kappa_d^{-1}\mu-\Delta)^{-\frac{1}{2}}(\mu-\Delta)^{-1}\rho |(\mu-\Delta)^{-1}f |\|_p \\
& + 2C_1\sqrt{l}\|(|b_n|^{\frac{1}{p}}+1)(\kappa_d^{-1}\mu-\Delta)^{-\frac{1}{2}}(\mu-\Delta)^{-1}\rho |\nabla(\mu-\Delta)^{-1}f|\|_p \\
& + \|(|b_n|^{\frac{1}{p}}+1)(\kappa_d^{-1}\mu-\Delta)^{-\frac{1}{2}}|(\mu-\Delta)^{-1}\rho f|\|_p,
\end{align*}
and so
$
I_2 \leq C_2 c' l I_1 + 2C_1c' \sqrt{l} I_3 + \|(|b_n|^{\frac{1}{p}}+1)(\kappa_d^{-1}\mu-\Delta)^{-\frac{1}{2}}|(\mu-\Delta)^{-1}\rho f|\|_p.
$

Assembling the above estimates, we conclude that there exists a constant $C>0$ such that, for any $\varepsilon_0>0$, there exists a sufficiently small $l>0$ such that
\begin{align*}
 (1-\varepsilon_0)I_0 & \leq \|(|b_n|^{\frac{1}{p}}+1) \nabla (\mu-\Delta)^{-1} \rho f\|_p \\
& + C\varepsilon_0 \biggl[\bigl\|(|b_n|^{\frac{1}{p}}+1)(\mu-\Delta)^{-1} \rho f\bigr\|_p + \bigl\|(|b_n|^{\frac{1}{p}}+1)(\kappa_d^{-1}\mu-\Delta)^{-\frac{1}{2}}|(\mu-\Delta)^{-1}\rho f|\bigr\|_p \biggr]. 
\end{align*}
Put $f:=|b_n|^{\frac{1}{p'}}h$, $h \in C_c$. Then, using $\|T_p(b_n)\|_{p 
\rightarrow p} \leq m_d c_p\delta$ (cf.\,section \ref{sect_1}), and applying \eqref{est_prop21} to the terms in  brackets $[\;\;]$, we obtain:
 For any $\varepsilon>0$ there exists $l>0$ such that, uniformly in $n$,
\begin{align}
\label{before_T_est}
\big\|\rho (|b_n|^{\frac{1}{p}}+1) \nabla (\mu-\Delta)^{-1}|b_n|^{\frac{1}{p'}}h \big\|_p < (1+\varepsilon) m_d c_p \delta \|\rho h\|_p,
\end{align}
so
\begin{equation}
\label{T_est_weight}
\|\rho\, T_p(b_n) h\|_{p} \leq (1+\varepsilon) m_d c_p \delta \|\rho h\|_p.
\end{equation}
We select $\varepsilon>0$ so that $(1+\varepsilon) m_d c_p \delta<1$. (Recall that $m_d c_p \delta<1$.)

Arguing as in the proof of \eqref{before_T_est} but taking $f:=h$ we find a constant $M_1<\infty$ such that
\begin{equation}
\label{M_est}
\|\rho\, |b_n|^{\frac{1}{p}}\nabla(\mu-\Delta)^{-1} h\|_p \leq M_1\|\rho h\|_p, \quad \text{ uniformly in $n$.}
\end{equation}

Also, we find a constant $M_2<\infty$ such that
\begin{equation}
\label{M_est_0}
\|\rho (\mu-\Delta)^{-1} |b_n|^{\frac{1}{p'}}h\|_\infty \leq M_2\|\rho h\|_p, \quad \text{ uniformly in $n$.}
\end{equation}
Indeed, using \eqref{est_1} with $f:=|b_n|^{\frac{1}{p'}}h$, we obtain
\begin{align*}
\|\rho (\mu-\Delta)^{-1} |b_n|^{\frac{1}{p'}}h\|_\infty & \leq C_2l\|(\mu-\Delta)^{-1}\|_{\infty \rightarrow \infty}\|\rho (\mu-\Delta)^{-1} |b_n|^{\frac{1}{p'}}h\|_\infty \\
&+ 2C_1\sqrt{l}\|(\mu-\Delta)^{-1}\|_{p \rightarrow \infty}\|\rho \nabla (\mu-\Delta)^{-1}|b_n|^{\frac{1}{p'}}h\|_p \\
&+ \|(\mu-\Delta)^{-\frac{1}{2}-\frac{1}{2q}}\|_{p \rightarrow \infty} \|Q_p(q)\rho h\|_p,
\end{align*}
where $\|\rho \nabla (\mu-\Delta)^{-1}|b_n|^{\frac{1}{p'}}h\|_p \leq (1+\varepsilon)m_dc_p\delta\|\rho h\|_p$ by \eqref{before_T_est}, and $\|(\mu-\Delta)^{-\frac{1}{2}-\frac{1}{2q}}\|_{p \rightarrow \infty}<\infty$ because $p>d-1$ and $q$ can be chosen arbitrarily close to $p$. Select $l>0$ so that $C_2l\mu^{-1}<1$. \eqref{M_est_0} follows.

Now, \eqref{theta_repr}  
combined with \eqref{T_est_weight}-\eqref{M_est_0} yields  \eqref{j_1_w}.

\end{document}